\documentclass[letterpaper,12pt]{amsart}
\usepackage{hyperref}
\usepackage{latexsym}
\usepackage{amssymb}
\usepackage[cp850]{inputenc}
\usepackage{epsfig}
\usepackage{amsthm}
\usepackage{amscd}
\usepackage{amsmath}
\usepackage{amsfonts}
\usepackage{graphics}
\usepackage[all]{xy}
\usepackage{dsfont}
\usepackage{epigraph}

\newtheorem{theorem}{Theorem}
\newtheorem{lemma}[theorem]{Lemma}
\newtheorem{corollary}[theorem]{Corollary}
\newtheorem{prop}[theorem]{Proposition}
\newtheorem{claim}[theorem]{Claim}

\newtheorem*{theorem*}{Theorem}
\newtheorem*{corollary*}{Corollary}
\theoremstyle{definition}

\newtheorem{definition}[theorem]{Definition}
\newtheorem*{remark*}{Remark}
\newtheorem{question}[theorem]{Question}
\newtheorem{observation}[theorem]{Observation}
\newtheorem*{definition*}{Definition}
\newtheorem*{example*}{Example}

\numberwithin{theorem}{section}

\newcommand{\BC}{\mathbb C} 
\newcommand{\BR}{\mathbb R} 
\newcommand{\BN}{\mathbb N} \newcommand{\BQ}{\mathbb Q}
 \newcommand{\BZ}{\mathbb Z}
 \newcommand{\BT}{\mathbb T}

\newcommand{\CC}{\mathcal C} \newcommand{\calD}{\mathcal D}

\newcommand{\CO}{\mathcal O} 
 
\newcommand{\CS}{\mathcal S} \newcommand{\CT}{\mathcal T}
 
 \newcommand{\CX}{\mathcal X}
 
\newcommand{\FG}{\mathfrak g}
\newcommand{\aut}{\textup{Aut}(F_n)}

\newcommand{\wt}{\widetilde}
\newcommand{\wh}{\widehat}
\newcommand{\nid}{\noindent}

\newcommand{\tup}{\textup}

\DeclareMathOperator{\trace}{Tr}

\DeclareMathOperator{\rank}{rank}

\DeclareMathOperator{\length}{length}

\newcommand{\comment}[1]{}

\usepackage{etoolbox}

\makeatletter
\patchcmd{\epigraph}{\@epitext{#1}}{\itshape\@epitext{#1}}{}{}
\makeatother

\begin{document}
\title  [Every infinite order mapping class has an infinite order action]  {Every Infinite order mapping class has an infinite order action on the homology of some finite cover}
\author   {Asaf Hadari}
\date{\today}
\begin{abstract} We  prove the following well known conjecture: let $\Sigma$ be an oriented surface of finite type whose fundamental group is a nonabelian free group. Let $\phi \in \tup{Mod}(\Sigma)$ be a an infinite order mapping class. Then there exists a finite solvable cover $\wh{\Sigma} \to \Sigma$, and a lift $\wh{\phi}$ of $\phi$ such that the action of $\wh{\phi}$ on $H_1(\wh{\Sigma}, \BZ)$ has infinite order. Our main tools are the theory of homological shadows, which was previously developed by the author, and Fourier analysis.\end{abstract}
\maketitle

\tableofcontents
\vspace {10mm}

\section{Introduction}

\smallskip

Let $\Sigma$ be an oriented surface of finite type, and let $\tup{Mod}(\Sigma)$ be its mapping class group. The finite dimensional representation theory of $\tup{Mod}(\Sigma)$ has been the subject of much study, and much about it remains mysterious. For instance, it is not generally known whether the groups $\tup{Mod}(\Sigma)$ are linear. On the other hand, these groups have extensive collections of finite dimensional representations.    

The largest and most tractable collection of representations are the \emph{homological representations} constructed as follows. The first of these is the standard  homological representation $\alpha: \tup{Mod}(\Sigma) \to \tup{GL}(H_1(\Sigma))$ given by the action on first homology. The kernel of this representation is called the \emph{Torelli group}, and is a an important group that has been the object of much study. More generally, pick a base point $\beta \in \Sigma$ and let $\Sigma'$ be the surface $\Sigma$ punctured at $\beta$. If $\pi: \Sigma_0 \to \Sigma$ is a finite cover, there is a finite index subgroup $M_0 \leq \tup{Mod}(\Sigma')$ of mapping classes that lift to $\Sigma_0$. This gives a representation $M_0 \hookrightarrow \tup{Mod}(\Sigma_0) \to \tup{GL}(H_1(\Sigma_0))$, that can be induced to a representation $\rho_\pi$ of $\tup{Mod}(\Sigma)$. Thus, for each finite cover of $\Sigma$, we construct a finite dimensional representation of $\tup{Mod}(\Sigma')$. 

Homological representations are a very rich family of representations. For example, in \cite{GLLM} Grunewald, Larsen, Lubotzky and Malestein show that the mapping class group of a once punctured surface has several infinite families of arithmetic groups as images of homological representation (In \cite{GL}, Grunewald and Lubotzky show a similar result for homological representations of $\aut$, which are constructed in a similar way). Looijenga proved a similar result for abelian covers in \cite{Looi}.

These representations contain a great deal of information about the mapping class group. For instance, Putman and Wieland exhibited a beautiful connection between properties of homological representations and the virtual first Betti number of mapping class group (\cite{PW}). McMullen studied the structure of genus $0$ homological representations extensively and exhibited strong connections to the study of moduli spaces of Rieman surfaces (\cite{Mcm0}). 

Furthermore, due to the fact that homological representations are the only class of representations of $\tup{Mod}(\Sigma)$ whose images are well understood, they often come up in applications where representations of mapping class group with large images are needed. One such application is the study of properties of random elements of subgroups of $\tup{Mod}(\Sigma)$ (see for example \cite{LMeiri}, \cite{LMeiri2}, \cite{MalesSouto}, \cite{Riv}).

Aside from containing information about the group as a whole, these representations contain information about individual elements. It is a simple exercise to show that for any nontrivial $\phi \in \tup{Mod}(\Sigma)$, there is a cover $\pi$ in which $\rho_\pi(\phi)$ is non trivial. In fact Koberda and later Koberda and Mangahas showed a much stronger result - that homological representations can detect the Nielsen-Thurston classification of a mapping class (\cite{Kober}, \cite{KoberMang}).

Despite being a  well studied family, homological representations still remain quite mysterious and many surprisingly basic and natural questions about them remain unanswered. One such well known question is the following, whose positive answer is a well known conjecture.   

\begin{question}\label{question} Suppose $\phi \in \tup{Mod}(\Sigma)$ has infinite order. Is it possible to find a cover $\pi$ such that $\rho_\pi(\phi)$ has infinite order? 
\end{question}

McMullen considered a stronger version of  question \ref{question}.  Let $\phi$ be a pseudo-Anosov mapping class with dilatation $\lambda$. For any $\pi$, let $\sigma(\pi)$ be the spectral radius of $\rho_\pi(\phi)$.  It is easy to show that $\log(\sigma(\pi)) \leq \lambda$. For pseudo-Anosovs with orientable foliations, it's true that $\lambda = \log \sigma_{Id}$, and for a while it was conjectured that for any pseudo-Anosov map: $\lambda = \sup \log \sigma_\pi$ where the supremum is taken over all finite covers. In \cite{Mcm}, McMullen showed that this is not the case for many pseudo-Anosov maps.  He then asked the following question, whose positive answer is also a folk conjecture. 

\begin{question} In the notation above, is $ \sup \log \sigma_\pi > 0$? In other words, can one always find a finite cover where the action on homology has an eigenvalue off of the unit circle?
\end{question}

In this paper we provide a positive answer to question \ref{question} for mapping classes of punctured surfaces. This provides some positive evidence towards McMullen's conjecture. 

\begin{theorem}\label{theorem1} Let $\Sigma$ be an oriented surface of finite type whose fundamental group is free. Let $\phi \in \tup{Mod}(\Sigma)$ be an infinite order mapping class. Then there exists a finite solvable cover $\Sigma_0 \to \Sigma$, and a lift $\phi_0$ of $\phi$ such that the action of $\phi_0$ on $H_1(\Sigma_0)$ has infinite order. 
\end{theorem}

\nid We also prove a related theorem for automorphisms of free groups. 
\begin{theorem} \label{theorem3} Let $\phi \in \tup{Out}(F_n)$ be a fully irreducible automorphism. Then for any representative $f$ of $\phi$ in $\aut$  there exists an $f$-invariant finite index subgroup $K \lhd F_n$ such that the action of $f|_K$ on $H_1(K)$ has infinite order, and $F_n / K$ is solvable.

\end{theorem}

Note that the restriction that $\Sigma$ have a free fundamental group is rather technical. We conjecture that our proof can be strengthened to cover the case of closed surfaces. 

\bigskip

\nid \textbf{Acknowledgements.} We wish to thank Benson Farb, Thomas Koberda, and Justin Malestein for helpful discussions and comments about earlier versions of this paper. 

\subsection{Sketch of proof.}
We sketch the proof of Theorem \ref{theorem3}, which implies Theorem \ref{theorem1} after some simple reductions. Part of the difficulty in the proof arises from the fact that we are searching for a finite cover with certain properties, and that the set of all such finite covers is extremely large. To circumvent this problem, we will always try to look for abelian covers, which can be parametrized as rational points on a torus, and will end up taking a sequence of such covers. 
\medskip

\nid \textbf{Step 1.}  Let $f$ be as in the statement of the theorem. We begin by attempting to find a finite abelian cover satisfying the conclusion of the theorem. In order to find such a cover, we consider an infinite abelian cover that contains information about all finite abelian covers. Let $R = \bigvee_1^n S^1$ be a wedge of $n$ circles. We have that $\pi_1(R) = F_n$. We pick a continuous map $\varphi: R \to R$ inducing the automorphism $f$. Let $\wt{R} \to R$ be the universal abelian cover, that is the cover corresponding to $[F_n, F_n]$. The space $\wt{R}$ is an $n$-dimensional grid. 

Let $G = C_1(\wt{R}, \BZ)$ be the space of $1$-chains in $\wt{R}$, that is - formal combinations of edges in $\wt{R}$. The reason we work with $1$-chains as opposed to homology classes is that the space $G$ has a particularly nice aglebraic description. Namely, if we set $H = H_1(F_n, \BZ)$ to be the abelianization of $F_n$, then $G \cong \BZ[H]^n$. Furtheromore, the map $\varphi$ can be lifted to a map $\wt{\varphi}$ that acts on $G$. This action also has a particularly nice algebraic description in the case that $f$ acts trivially on $H$ (if $f$ does not act trivially on $H$ then we can either replace it with a power of itself or conclude the proof). Indeed, the action of $\wt{\varphi}$ on $G$ is given by multiplication by a matrix $A_\varphi \in M_n(\BZ[H])$. 

The matrix $A_\varphi$ contains information about the action of $f$ on every finite abelian cover. Any finite abelian cover is given by a map $\psi: H \to (S^1)^n$ whose image is a rational point of the torus. Denote such a cover by $R_\psi$. Every such map $\psi$ can also be applied to the matrix $A_\varphi$ to get a matrix $\psi(A_\varphi) \in M_n(\BC)$. By definition, it's simple to see that this matrix gives the action of a lift of $\varphi$ on a subspace of $C_1(R_\psi, \BZ)$, which can be identified with the twisted chain space $C_1(R, \BC_\psi)$. 

We show that if this action has eigenvalues off of the unit circle, then so does the action of a lift of some power of $f$ on $H_1(R_\psi, \BZ)$. Thus, it is enough to show that for some $\psi$, the matrix $\psi(A_\varphi)$ has eigenvalues off of the unit circle. One way to do this is to find a $\psi$ such $|\psi(\trace(A_\varphi))| > n$. By isometry of the Fourier transform on $\BZ^n$, it is enough to show that $\trace(A_\varphi) \in \BZ[H]$ has at least $n+1$ summands.  \\

\nid \textbf{Step 2.}  If $\trace(A_\varphi)$ has at least $n+1$ summands, then we are done. Suppose it does not. We now turn to finding potential candidates for summands that will appear in the traces corresponding to finite covers. These will be summands that cancel out in homology, but will not cancel out in the homology of some cover. To do this, we use the technology of homological shadows. Given a word $w \in F_n$, the \emph{homological shadow of w} or $\CS w$ is the set of vertices of $\wt{R}$ through which a lift of $w$ passes. Since the vertices of $\wt{R}$ correspond to $H$, we can think of $\CS w$ as a subset of $H \otimes \BR \cong \BR^n$. In a previous paper, we showed that there exists an $n$-dimensional convex polytope $\CS f \in \BR^n$ such that for any word $w$ with infinite $f$ order, $\lim_{k \to \infty} \frac{1}{k}\CS f^k(w) = \CS f$. 

As an $n$-dimensional polytope, $\CS_f$ has at least $n+1$ vertices.  These vertices correspond to the summands that we are seeking. Indeed, we show that if for every vertex $v$, if we can find a cover $R_v \to R$ to which $f$ lifts, such that the action of $f$ on $H_1(R_v, \BZ)$ either has infinite order, or has finite order and a term projecting to $\BR v$ appears in $\trace (A_{\varphi_v})$ (where $\varphi_v$ denotes a lift of $\varphi$ to the cover $R_v$), then we are done. If the above condition holds in some cover $R_v$, we say that $v$ is \emph{enfeoffed} in that cover (the word enfeoff means to grant a fief. The idea here being that an enfeoffed vertex is responsible for a portion of the trace). \\

\nid \textbf{Step 3.} Now, the proof reduces to the following problem - given a vertex $v$ of $\CS f$, find a cover in which it is enfeoffed. This is the most difficult and involved part of the proof. We do this in two stages.  
\begin{enumerate} 
\item We assign a measure of the complexity of the cancellation at the vertex $v$ in the following way. Let $w \in F_n$, and let $R_0 \to R$ be a normal cover corresponding to the group $K < [F_n, F_n]$. Pick $k$, and consider the element of $C_1(R_0, \BZ)$ obtained by lifting $f^k(w)$ to $R_0$. This chain can be thought of as an element of $\BZ[D]^n$, where $D$ is the deck group of the cover. From this chain, remove all summand not supported at elelments of $D$ that project to $kv$. Call the remaining chain the \emph{$v$ part of $f^k(w)$ in $R_0$}. We say that $v$ is $R_0$ enfeoffed if $v$ part of $f^k(w)$ in $R_0$ is not equal to zero for some $w$ and infinitely many values of $k$. We now let $R_1, R_2, \ldots$ be the sequence of free $i$-step nilpotent covers of $R$. We show in lemma \ref{nilenf} that there is some $i$ such that $v$ is $R_i$ enfeoffed. The minimal such number $i$ is the complexity we wish to measure. It measures how deep into the lower central series we need to go before the vertex $v$ stops being cancelled. We say that $v$ is \emph{$i^{th}$ level nilpotent enfeoffed}.   The technique used to prove this lemma is similar to our previous use of homological shadows. 

\item In Lemma \ref{nilpup}, we show that if $v$ is $i^{th}$ level nilpotent enfeoffed, then we can pass to some finite abelian cover where it is $j^{th}$-level nilpotent enfeoffed for some $j < i$. The techniques used here are mostly fourier-analytic. Continuing this process inductively, we find a cover where $v$ is $1^{st}$ level nilpotent enfeoffed, which is the same as being enfeoffed. 
\end{enumerate}

\section{Proof}

\subsection{First reductions}
Suppose $\phi \in \tup{Mod}(\Sigma)$ is a mutitwist about the muticurve $C$. The action of $\phi$ on $H_1(\Sigma)$ is given by a product of commuting transvections. This is true for the action on the homology of any cover.  Such a product is either trivial or infinite order. Thus, if we pick a characteristic cover $\Sigma' \to \Sigma$ such that the action of $\phi$ on $H_1(\Sigma')$ is non-trivial, then it will have infinite order. 

By the Nielsen-Thurston classification of mapping class groups, it remains to prove the result for a mapping class $\phi$ such that $\phi^k$ restricts to a pseudo-Anosov diffeomorphism of a subsurface $S \subset \Sigma$ and some k.  Such a diffeomorphism induces a well defined element $\phi_* \in \tup{Out}(\pi_1(S))$, which has a train track representative with Perron-Frobenius transition matrix (see \cite{BeH} for definitions). Since $\pi_1(S)$ is a free factor in $\pi_1(\Sigma)$, Theorem \ref{theorem1} follows from the following theorem.

\begin{theorem}\label{theorem2} Let $\phi \in \tup{Out}(F_n)$ be an outer automorphism that has a train track representative with Perron-Frobenius transition matrix. Then for any representative $f$ of $\phi$ in $\aut$ there exists a characteristic finite index subgroup $K \lhd F_n$ such that the action of $f$ on $H_1(K)$ has infinite order and $F_n /K$ is solvable.  
\end{theorem}

Note that Theorem \ref{theorem3} is a direct corollary of Theorem \ref{theorem2}. The reset of the paper will consist of a proof of Theorem \ref{theorem2}

\subsection{Notation and preliminaries}\label{sec2.1}
Pick an automorphism $f$ that projects to $\phi$. Let $(\Gamma, \beta)$ be a graph with fundamental group $F_n$ (which we will specify later on in the proof) together with a chosen base point $\beta \in V(\Gamma)$.  Denote $m = \#E(\Gamma)$. Fix, once and for all, an orientation on each of the edges of $\Gamma$.

Let $\varphi: (\Gamma, \beta) \to (\Gamma, \beta)$ be a continuous function that induces the automorphism $f$. This map can be taken to send any edge to a concatenation of edges.

Let $ \wt{\Gamma} \to \Gamma$ be the universal abelian covering of $\Gamma$, that is - the covering associated to the subgroup $[F_n, F_n]$. Denote $H = H_1(\Gamma)$ (which we will think of as a multiplicative group), and $G = C_1(\wt{\Gamma}, \BZ)$, the group of $1$-chains in $\wt{\Gamma}$. The group $G$ has an obvious action of $H$, which turns it into a $\BZ[H]$ module. In fact, we have the $\BZ[H]$-module isomorphism $G \cong \BZ[H]^{E(\Gamma)}$.

\subsubsection{The lifted action of $\varphi$ on $G$}

The automorphism $f$ acts on $H$. This extends  to an action on $\BZ[H]$ which we denote by $\sigma$. 

Pick a preferred lift $\wt{\beta}$ of $\beta$, and lift $\varphi$ to a continuous map $\wt{\varphi}: (\wt{\Gamma}, \wt{\beta}) \to (\wt{\Gamma}, \wt{\beta})$. The map $\wt{\varphi}$ acts on $G \cong \BZ[H]^{E(\Gamma)}$. This action is given by $$\wt{\varphi}_* \left(\begin{array}{c} h_1 \\ \vdots \\ h_m \end{array} \right) = A_\varphi  \left(\begin{array}{c} \sigma(h_1) \\ \vdots \\ \sigma(h_m) \end{array} \right)  $$

\nid for some matrix $A_\varphi \in M_m(\BZ[H])$. \\

We wish to describe the matrix $A_\varphi$ explicitly. Since our proof uses several variants of this matrix, we choose to construct it in a somewhat indirect way via a directed graph (or digraph) called the \emph{transition graph of} $\varphi$. 

\subsubsection{The transition graph of $\varphi$ and the matrix $A_\varphi$} Let $\CT = \CT_\varphi$ be the following directed graph. Set $V(\CT) = E(\Gamma)$. Let $e \in V(\CT)$. Write $\varphi(e) = e_1 \ldots e_s$, where each $e_i$ is an edge of $\Gamma$. Add  $s$ edges to $\CT$ emanating from $e$, where the $i^{th}$ edge, which we denote $\eta_i$, connects $e$ to the vertex $e_i$.  

Let $x_e \in \BZ[H]^{E(\Gamma)} \cong G$ be the element with $1$ in the $e$-coordinate, and $0$ in all others. We think of $x_e$ as an oriented edge in $\wt{\Gamma}$. Let $\wt{e_1} \ldots \wt{e_s}$ be the path $\wt{\varphi}{x_e}$. Write this path as $\wt{w}\wt{g_i}^{\pm 1} \wt{u}$, where $\wt{g_i}$ is the edge corresponding to $\wt{e_i}$, traversed in the positive direction. We associate the following objects to the edge $\eta_i$.
\begin{enumerate}
\item The \emph{sign} of $\eta_i$, or $\mathfrak{s}(\eta_i) \in \{1,-1 \}$, defined to be $1$ if $\wt{g_i} = \wt{e_i}$ and $-1$ otherwise.  
\item The edge $\wt{g_i}$ can be thought of as an element in $\BZ[H]^{E(\Gamma)}$ that is equal to $\mathfrak{t}(\eta_i)  \in H$ in the $g_i$ coordinate, and $0$ everywhere else. We call $\mathfrak{t}(\eta_i)$ the \emph{translation of } $\eta_i$.

\end{enumerate}

To a cycle $\gamma = \eta_1 \ldots \eta_k$ in $\CT$ we associate the following objects. 

\begin{enumerate} 
\item The \emph{sign of } $\gamma$, or $\mathfrak{s}(\gamma) = \prod_{i=1}^k \mathfrak{s}(\eta_i)$. 
\item The \emph{translation of $\gamma$} or $\mathfrak{t}(\gamma) = \prod_{i=1}^k \sigma^{i-1}\mathfrak{t}(\eta_i)$ and the \emph{normalized translation of $
\gamma$} or $\mathfrak{t}_n(\gamma) = \frac{1}{k}\mathfrak{t}(\gamma)$, where the latter is thought of as an element of $H \otimes \BR$. 

\end{enumerate}

Note that throughou this paper, we employ the word cycle in a somewhat nonstandard way to mean a sequence of edges originating and terminating at the same point, \emph{not up to cyclical re-ordering of the edges}. Our cycles will always be based cycles. While this does not affect the definitions given above, it will come up several times in the sequel.

We are now able to give a description of $A_\varphi$, that follows directly form the definitions.  Given $e, e' \in E(\Gamma)$, we have that $(A_\varphi)_{e,e'} = \sum \mathfrak{s}(\eta)\mathfrak{t}(\eta)$, where the sum is taken over all edges $\eta$ of $\CT$ connecting $e$ to $e'$. Notice that in this sum, we are viewing $H$ as a subset of $\BZ[H]$ in the obvious way. 

\subsubsection{Matrices associated to subgraphs of $\CT$ and to covers} We will require two variants of the matrix $A_\varphi$. The first variant is obtained by considering subgraphs of $\CT$. Let $\CT'$ be a subgraph of $\CT$. We construct a matrix $A_\varphi[\CT']$ by setting:

$$(A_\varphi[\CT'])_{e, e'} =  \sum_{\eta = (e, e') \in E(\CT')} \mathfrak{s}(\eta)\mathfrak{e}(\eta)$$

The second variant involves covers. Let $\pi_0: (\Gamma_0, \beta_0) \to (\Gamma, \beta)$ be a finite cover to which $\varphi$ can be lifted, and let $\varphi_0$ be a lift of $\varphi$  such that $\varphi_0(\beta_0) = \beta_0$. Let $\CT_0$ be the directed graph assigned to $\varphi_0$. Then the graph $\CT_0$ is a finite cover of $\CT$, where the covering map respects the directionality of edges.  If $\CT'$ is a subgraph of $\CT$, let $\CT'_0$ be its lift to $\CT_0$. We will use the notation $A_\varphi[\CT', \pi_0]$  for the matrices attached to $\CT'_0$.

\subsection{Abelian covers and $\trace(A_\varphi)$}
\subsubsection{The polynomial $\Delta$ and specializations of $\Delta$ and $A_\varphi$}
Using the notation of the previous section, let $\Delta =\Delta_\varphi(t) = \det (tI - A_{\varphi}) \in \BZ[H][t]$ be the characteristic polynomial of $A_\varphi$. Given a homomorphism $\psi: H \to S^1$, we  define the \emph{specialization of  $\Delta$ at $\psi$}, or $\Delta_\psi$ by taking the image of $\Delta$ under the map $\BZ[H][t] \to \BZ[\psi(H)][t]$. Similarly, we define $(A_\varphi)_\psi \in M_m(\BC)$ to be the image of $A_\varphi$ under this map. \\

\nid Note: we use the letter $A$ for $A_\varphi$ and the letter $\Delta$ because they are variants of the Alexander matrix and the Alexander polynomial (which is often denoted $\Delta$). Like the Alexander polynomial, the polynomial $\Delta$ is useful for studying the action on $\varphi$ on homology with twisted coefficients (See \cite{Mcm2} for definitions of the Alexander polynomial and its uses in studying homology with twisted coefficients.)

\subsubsection{Chains with twisted coefficients}

Suppose that $\psi: H \to S^1 \subset \BC^*$ has finite image. Let $\Gamma_\psi \to \Gamma$ be the finite cover corresponding to the kernel of $\psi$. Then the complex $C_*(\Gamma_\psi)$, and the space $\BC$ are both $\BZ[H]$ modules (where the action on $\BC$ is given by $\psi$). Let $C_*(\Gamma, \BC_\psi)$ be the complex  $C_*(\Gamma_\psi, \BZ) \otimes_{\BZ[H]} \BC$.  \\

 The space $C_1(\Gamma, \BC_\psi)$ can be naturally identified with the subspace of $C_1(\Gamma_\psi,  \BC)$ on which the action of $\BZ[H]$ factors through $\psi$, that is - the set of all $x \in C_1(\Gamma_\psi, \BC)$ such that for any $\gamma \in F_n$, $\gamma(x) = \psi(\gamma) x$.  Suppose that $\sigma$, the action of $f$ on $H$ is trivial. Suppose further that $\varphi$ can be lifted to $\Gamma_\psi$. Then the action of the lift of $\varphi$ on the space $C_1(\Gamma_\psi, \BC_\psi)$ is given in some basis by multiplication by the matrix $(A_\varphi)_\psi$, and its characteristic polynomial is $\Delta_\psi$.

\subsubsection{Fourier transforms on Abelian groups} Before we proceed, we need to recall some standard facts about the Fourier transform $\BZ^n$.
Recall that $\BT = (S^1)^n$ is the group of characters of $\BZ^n$. Let $f \in L^1[\BZ^n]$. We define a function $\wh{f} \in L^1(T)$ by the rule: 

$$\wh{f}(\xi) = \sum_{x \in \BZ^n} f(x) \overline{\xi}(x)$$

\nid We summarize the facts we need in the following proposition. 
\begin{prop}\label{fourier} The Fourier transform satisfies the following properties. 
\begin{enumerate}
\item It is invertible, with inverse given by:
$$f(x) = \int_{T} \wh{f}(\xi) \xi(x) d\nu(x) $$
where $\nu$ is the Haar measure on $T$. 
\item It extends to an isometry $L^2(\BZ^n) \to L^2(\BT, \nu)$.
\item Given $f, g \in L_1(\BZ^n)$, one has that: 
$$\wh{f\cdot g} = \wh{f} \cdot \wh{g} $$ 
\end{enumerate}
\end{prop}

\subsubsection{Passing to abelian covers} Suppose $f$ acts trivially on $H$. Denote by  $t_\varphi = \trace{A_\varphi} \in L_2[H]$. The element $t_\varphi$ will be crucial for our proof of Theorem \ref{theorem2}, due to the following lemma. 
\begin{lemma} \label{lemma2.1}
Suppose $\|t_\varphi \|_2 > \sqrt{m}$. Then there exists a finite cover $\Gamma_0 \to \Gamma$, an integer $i$ and a lift $\varphi_0$ of $\varphi^i$ to $\Gamma_0$ such that the action of $\varphi_0$ on $H_1(\Gamma_0, \BC)$ has an eigenvalue of absolute value greater than $1$. 
\end{lemma}
\begin{proof}

By the second property in Proposition \ref{fourier}, there exists a $\psi \in \BT$ such that $(t_\varphi)_\psi > m$. Since $\wh{t_\varphi}$ is a continuous function, we can take this $\psi$ to have finite image. By replacing $f$ with a power of itself, we can assume that $\varphi$ lifts to $\Gamma_\psi$. Furthermore, by replacing it with a further power, we can assume that that we still have $(t_\varphi)_\psi > \sqrt{m}$. Thus, the action of the lift of $\varphi$ on $C_1(\Gamma, \BC_\psi)$ has an eigenvalue with absolute value $>1$. The following lemma will complete the proof. 

\begin{claim} Let $\Gamma$ be a graph with fundamental group $F_n$ and let $\varphi: \Gamma \to \Gamma$ be a continous function inducing the automorphism $f \in \aut$. Suppose that the action of $\varphi$ on $C_1(\Gamma,  \BC)$ has an eigenvalue with absolute value $>1$. Then the same is true for the action of $\varphi$ on $H_1(\Gamma, \BC)$. 
\end{claim} 
\begin{proof}
Let $U = C_1(\Gamma, \BC)$, and let $W \subset U$ be the subspace spanned by all closed paths in $\Gamma$. We have a natural identification $W \cong H_1(\Gamma, \BC)$. The space $U$ is spanned by elements of the form $1_e$, where $e$ ranges over all edges of $\Gamma$. Pick any norm $\|\cdot \|$ on $U$, and let $\lambda > 1$ be the spectral radius of the action of $\varphi$ on $U$. Then there exists an edge $e$ such that $$\lim \sup_{k \to \infty} \frac{1}{k} \log \|\varphi^k (1_e) \| = \log \lambda  $$
Denote $L_\lambda$ to be the direct sum of all generalized eigenspaces corresponding to eigenvalues with absolute value $\lambda$.  Setting $\nu_k = \frac{\varphi^k (1_e)}{ \|\varphi^k (1_e) \|}$, we have that the distance from $\nu_k$ to $L_\lambda$ goes to $0$ as $k \to \infty$. 
Note that $\varphi^k(1_e)$ corresponds to a path in $\Gamma$, and any such path can be closed to a loop using a bounded number of edges. Thus, the distance of $\nu_k$ from $W$ goes to $0$ as $k \to \infty$.  Therefore $L_\varphi \cap W \neq \{0 \}$. Since this space is $\varphi$-invariant, it contains an eigenvector with eigenvalue of absolute value $\lambda$.

\end{proof}

\end{proof}

\nid We have shown the following. 

\begin{prop}\label{prop1} Suppose that $f \in \aut$ is fully irreducible, $\Gamma$ is a graph with $\pi_1(\Gamma) \cong F_n$, and that $\varphi: \Gamma \to \Gamma$ is a continuous map inducing $\Gamma$. Suppose that there is a cover $\pi_0: \Gamma_0 \to \Gamma$ to which $\varphi$ lifts (to a function $\varphi_0$) such that  there is an integer $i$ such that $\varphi_0^i$ acts trivially on $H_1(\Gamma_0)$ and $\|\trace (A_{\varphi_0^i}) \|_2 > \sqrt{\#E(\Gamma_0)}$. Then there is a finite index subgroup $K < F_n$ such that $f(K) = K$, and such that the action of $f$ on $H_1(K)$ has infinite order. 
\end{prop}

In our proof, we will take many covers of $\Gamma$. This will cause the quantity $\#E(\Gamma_0)$ to grow in a manner that is difficult to control. In order to deal with this, we will require a somewhat stronger version of Proposition \ref{prop1}, whose condition requires a bounded amount of checking. For this version, we fix $\Gamma = \bigvee_{i=1}^n S^1$,  a wedge of $n$ circles.  The edges of $\Gamma$ are a basis of $F_n$.  
\begin{corollary}\label{cor1}   Let $f$, $\Gamma$, $\varphi$ be as above. Suppose that there is regular a cover $\pi_0: \Gamma_0 \to \Gamma$ to which $\varphi$ lifts (to a function $\varphi_0$) such that  there is an integer $i$ such that $\varphi_0^i$ acts trivially on $H_1(\Gamma_0)$. Let $\calD$ be the deck group of the cover. If the support of $\trace(A_{\varphi_0^i})$ in $H^1(\Gamma_0)$ intersects at least $n+1$ $\calD$-orbits, then there is a finite index subgroup $K < F_n$ such that $f(K) = K$, and such that the action of $f$ on $H_1(K)$ has infinite order.
\end{corollary}

\begin{proof}
If $f$ acts with infinite order on $H_1(\Gamma_0)$, we are done. Otherwise, by replacing it with a power of itself, we can assume that the action of $f$ on $H_1(\Gamma_0)$ is trivial. The map $\varphi_0$ induces a map $V(\Gamma_0) \to V(\Gamma_0)$. The set $V(\Gamma_0)$ can be identified in the obvious way with $\calD$, and the action of $\varphi_0$ on this set is given by the action of $f$ as an element of $\tup{Aut}(\calD)$. Thus, by replacing $f$ with a further power of itself, we can assume that it fixes every vertex of $\Gamma_0$. This gives that for any $s \in \calD$, the maps $s \circ \varphi_0$, $\varphi_0 \circ s$ have the same action on $G$. Thus, we have that $\trace(A_{\varphi_0})$ is $\calD$-invariant. 

Given a $\calD$ orbit $\CO \subset H_1(\Gamma_0)$, let $t_\CO \in \BZ[H_1(\Gamma_0)]$ be the sum of all the terms in $\trace(A_{\varphi_0})$ contained in $\CO$. By invariance, $\|t_\CO \|_1$ is divisible by $|\calD| = \deg(\pi_0)$.  Thus, the condition in the statement of the corollary gives that $\| \trace(A_{\varphi_0})\|_1 > (n+1) |\calD| > n |\calD| = \#E(\Gamma_0)$. Applying Proposition \ref{prop1} now gives the result. 
\end{proof}

\nid From this point on, we fix $\Gamma = \bigvee_{i=1}^n S^1$. 

\subsection{Extremal subgraphs and homological shadows} \label{sec2.2}

\subsubsection{Homological shadows of words and paths} Given a path $p$ in $\Gamma$ originating at $\beta$, let $\wt{p}$ be the lift of $p$ originating at $\wt{\beta}$. Define the  $S(p) \in G$ to be $S(p) = \sum c_i \wt{e_i}$, where $c_i$ is the number of times $\wt{p}$ passes through the edge $\wt{e_i}$ \emph{in either direction} (so we always have that $c_i \geq 0$).  Given a, not necessarily reduced, word $w$ in the basis above, let $S^r(w) = S(\overline{w})$ where $\overline{w}$ is the reduced word corresponding to $w$ and $\overline{w}$ is thought of as a path in $\Gamma$. 

In the above situations, define the \emph{shadow of $p$} to be $\CS(p) = \tup{supp} S(p) \subset H \hookrightarrow H_1(F_n, \BR)$, where $\tup{supp}$ is the support function. Similarly, define the \emph{shadow of $w$}, or $\CS (w) \subset H_1(F_n, \BR)$ to be the support of $S^r(w)$. 

\subsubsection{Homological shadows of $f$ and $\varphi$} Suppose that $\sigma = I_n$. In \cite{Had}, the author proved the following Theorem. 

\begin{theorem} Suppose $f$ is has a train track representative with a Perron-Frobenius transition matrix and $\sigma = I_n$. There exists a convex polytope $\CS f \subset H_1(F_n, \BR)$ with rational vertices such that for any $w \in F_n$ with infinite $f$-orbit, $$\lim \frac{1}{k} \CS f^{k}(w) = \CS f $$
where the above limit is taken in the Hausdorff topology. We call $\CS f$ the \emph{shadow of $f$}. 
\end{theorem}

\nid While we are interested mainly in $\CS f$, for technical reasons we need to consider a different shadow that is attached to $\varphi$.   The same proof as the above theorem also shows the following. 
\begin{prop} Under the above conditions, there exists a convex polytope with rational vertices $\CS \varphi \subset H_1(F_n, \BR)$ called the \emph{shadow of $\varphi$} such that for any path $p$ in $\Gamma$ with infinite $\varphi$ orbit: 
$$\lim_{k\to \infty} \frac{1}{k} \CS \varphi^k p = \CS \varphi $$
Furthermore, $\CS \varphi$ is the convex hull of all elements of the form $\mathfrak{t}_n(\gamma)$, where $\gamma$ is a simple cycle in $\CT$.
\end{prop}

The shadow $\CS f$ is calculated in a similar manner to the shadow $\CS \varphi$, but uses the convex hull of $\mathfrak{t}_n(\gamma)$ where $\gamma$ is a cycle in the train track transition graph of a train track representative of $f$ (see \cite{Had} for details). This convex hull was considered by Fried in \cite{Friedzeta} for the case of an automorphism coming from a pseudo-Anosov diffeomorphism, and in a slightly different language by Dowdall, Kapovich and Leininger (\cite{DLK}, \cite{DLK2}) and separately by Algom-Kfir, Hironaka and Rafi (\cite{YEK}) for the general $\tup{Out}(F_n)$ case. In both cases, this convex hull gives a cross-section of the dual cone to a  fibered cone (see the above references for definitions) which is a $n+1$ dimensional cone when $\sigma = I_n$. Thus, we have the following. 
\begin{prop} The polytope $\CS f$ is $n$ dimensional. In particular, it has at least $n+1$ vertices. 
\end{prop}

\subsubsection{Group-like vertices} A vertex $v$ of $\CS \varphi$ will be called \emph{group-like} if it also a vertex of $\CS f$. These are the vertices that we will use in our proof. One technical problem we need to deal with is that vertices of $\CS f$ may not be vertices of $\CS \varphi$ at all. The following lemma deals with this problem. 

\begin{lemma} Suppose that $\sigma = I_n$. After perhaps replacing $f$ with $f^k$ for some number $k$, every vertex of $\CS f$ is a group-like vertex of $\CS \varphi$. 
\end{lemma}

\begin{proof} Notice first that for any $k$, $\CS f^k = k \CS f$. Let $\psi: (\Gamma, \beta) \to (\Gamma, \beta)$ be the continuous map associated to $f^k$. It is not necessarily true (and is most often false) that $\psi = \varphi^k$. We have that $\CS \psi \subseteq k \CS \varphi$, but equality does not need to hold. 

Let $F$ be a vertex of $\CS f$. Let $\omega \in H^1(F_n, \BR)$ be a dual element whose maximum in $\CS f$ is achieved at $F$. Let $\gamma_1, \ldots, \gamma_s$ be the set of all simple cycles in $\CT$ satisfying $\omega \big(\mathfrak{t}_n(\gamma_i) \big) > \omega|_F$. Any cycle $\gamma$ satisfying  $\omega \big(\mathfrak{t}_n(\gamma) \big) > \omega|_F$ has to contain such a cycle as a sub-cycle. 

Pick $k > \max_i \length (\gamma_i)$. Let $\psi$ be the map associated to $f^k$, and let $\CT_\psi$ be its associated graph. Any cycle of length $l$ in $\CT_\psi$ corresponds to a cycle of length $kl$ in $\CT$. No such cycle contains any of the $\gamma_i$'s. Thus, for any cycle $\gamma$ in $\CT_\psi$, we must have that $\omega(\mathfrak{t}_n(\gamma)) \leq k \omega|_F$. There must be at least one cycle where $\omega(\mathfrak{t}_n(\gamma)) = k \omega|_F$ (or else $F$ would not be a vertex of $\CS f$). Thus, $kF$ is a group-like vertex of $\CS \psi$. By choosing a sufficiently large value of $k$, this argument works for all vertices of $\CS f$.
\end{proof}

\subsubsection{Homological shadows and extremal subgraphs}  \label{subord}
\nid We assume once more that $\sigma = I_n$
\begin{definition} Let $F$ be a group-like vertex of $\CS \varphi$. Let $\CT_F \subset \CT$ be the union of all cycles $\gamma$ satisfying $\mathfrak{t}_n(\gamma) \in F$. 
\end{definition}

We call such a subgraph \emph{extremal}. Extremal subgraphs have nice properties, which are consequences of the following definitions, that are useful in the study of $\trace(A_\varphi)$.

\begin{definition}
Given $x = \sum a_h h \in \BZ[H]$, let $\tup{supp}(x) \subset H$ be the set of all $h$ such that $a_h \neq 0$. Given $x, y \in \BZ[H]$ we say that \emph{$x$ is subordinate to $y$} and write $x \preceq y$ if $y|_{\tup{supp}(x)} = x$. We say that $x$ and $y$ are \emph{separated} and write $x \| y$ if $\tup{supp}(x) \cap \tup{supp}(y) = \emptyset$.

\end{definition}

\begin{definition} Let $\CT' \subseteq \CT$ be a subgraph. For any $k$, let $t_k[\CT'] = \sum \mathfrak{s}(\gamma) \mathfrak{t}(\gamma)$ where the sum is taken over all cycles of length $k$ in $\CT'$. 
\end{definition}

\nid Note that we can calculate $t_k[\CT']$ in terms of $A_\varphi[\CT']$. Given $e \in E(\Gamma)$, let $x_e \in G$ be the vector with $1$ in the $e$-coordinate and $0$ in all the others. Then $$t_k[\CT'] = \sum_e x_e^T \big(A_\varphi[\CT'] \circ \sigma \big)^k x_e $$

\begin{definition} Let $\CT', \CT''$ be subgraphs of $\CT$. We say that $\CT'$ is subordinate to $\CT''$, and write $\CT' \preceq \CT''$ if for all sufficiently large $k$, $t_k[\CT'] \preceq t_k[\CT'']$. 

\end{definition}

\begin{definition} We say that two subgraphs $\CT', \CT''$ of $\CT$ are separated, and denote $\CT' \| \CT''$ if for all sufficiently large $k$, $t_k[\CT'] \tup{ } \| \tup{ } t_k[\CT'']$.
\end{definition}

\begin{definition} Let $\CT', \CT''$ be two subgraphs of $\CT$. If $\CT' \preceq \CT$, $\CT'' \preceq \CT$ and $\CT' \tup{ } \| \tup{ }\CT''$ then we write $\CT' \pitchfork \CT''$.  

\end{definition}

The first thing to note is the following 
 \begin{lemma} \label{pitchfork} If $v, w$ are two different disjoint group-like vertices of $\CS \varphi$ then $\CT_v \pitchfork \CT_w$.
 \end{lemma}

\begin{proof} 

\nid The result follows directly from the following claim: let $\gamma \subset \Gamma$ be a loop of length $k$, and let $u$ be a group-like face Then $\mathfrak{t}(\gamma) \in ku \iff \gamma \subset \CT_u$. 

The  if direction of this claim follows from the definition of $\CT_u$. To see the only if direction, suppose $\gamma' \subset \CT_u$ is a cycle of length $k$. Let $\omega \in H^1(F_n, \BR) \to \BR$ be a cohomology class whose maximum on $\CS \varphi$ is achieved at the face $u$, and which is constant on the face (such a class exists, because $u$ is a face of the convex polytope $\CS \varphi$). By maximality, $\omega(\mathfrak{t}(\gamma')) \leq k \omega(u)$. 

Suppose that $\omega(\mathfrak{t}(\gamma')) < k \omega(u)$. Write $\gamma' = e_1 \ldots e_r$. Since each edge of $\gamma'$, there are paths $\eta_1, \ldots, \eta_r$ in $\CT_u$ of lengths $k_1 -1, \ldots, k_r - 1$ such that for each $i$, $e_i \eta_i$ is a cycle and $\omega(\mathfrak{t}(e_i \eta_i)) = k_i \omega(u)$. Let $\eta$ be the cycle $\eta_1 \ldots \eta_r$. Let $\tau = \omega \circ \mathfrak{t}$. Then: 

$$\tau(\eta) = \sum_i \tau(\eta_i) = \sum_i [k_i \omega(u) - \tau(e_i)] = \omega(u) \sum_i k_i  - \omega(\gamma')$$

Notice that $\omega(u) \sum_i k_i = \length(\eta) \omega(u) + k \omega(u)$, and that $k \omega(u) - \tau(\gamma') \geq 0$. Thus, $\tau(\eta) > \length(\eta) \omega(u)$, which is a contradiction to the definition of $\omega$. Thus, $\tau(\gamma') = \length(\gamma') \omega(u)$.

\end{proof}

\nid Extremal subgraphs behave quite well with respect to covers.  

\begin{observation} If $\CT_v, \CT_w$ are extremal subgraphs corresponding to different verices, $\pi: \Gamma_0 \to \Gamma$ is a finite cover, and $\CT'_v, \CT'_w$ are the lifts of $\CT_v, \CT_w$ respectively, then $\CT'_v \pitchfork \CT'_w$. 
\end{observation} 

\subsubsection{Enfeoffed vertices}
\nid The reason we care about extremal subgraphs is that they allow us to break the study of $\trace(A_\varphi)$ into simple pieces. This idea is embodied in the following definition and observations. 

\
\begin{definition} Let $\CT_v \preceq \CT$ be an extremal subgraph. Let $\sigma$ be the action of $\varphi$ on $H_1(\Gamma)$. Suppose that $\sigma$ has finite order. We say that $\CT'$ is \emph{enfeoffed in } $\Gamma$ if $A_{\varphi}[\CT_v, \pi] \circ \sigma$ is not nilpotent map (note - the word enfeoffed is pronounced in-FEEF-d, and means to be granted a feudal fief. Here we use it to denote the intuition that if $\CT_v$ is enfeoffed in $\Gamma_0$, then $\CT_v$ is responsible for a portion of the support of  $\trace(A_{\varphi_0})$). Indeed, by replacing $f$ with a power of itself that acts trivially on homology, we have that $\CT_v$ is enfeoffed if and only if for infinitely many $i$, the support of $\trace(A^i_{\varphi_0})$ intersects $iv$. 
\end{definition}

\begin{observation} If $\CT_v$ is an enfeoffed extremal graph, and $\CT_v'$ is the lift of $\CT_v$ to some cover to which $\varphi$ lifts, then $\CT_v'$ is also enfeoffed. 
\end{observation}

\begin{observation} Suppose $\sigma$ has finite order, and that $\CT_v$ is an enfeoffed extremal subgraph. Then there exists a number $m$ such that $\sigma^m = Id$, and for any $k$, $\mathfrak{t}_{km}[\CT_v] \neq 0$. Thus, each such subgraph $\CT_v$ contributes terms to the support of  $A_{\varphi^k}$ for infinitely many $k$. In particular, if $\CT_1, \CT_2, \ldots, \CT_r$ are enfeoffed, pairwise separated extremal subgraphs, then for some $i$, such that $\sigma^i = I$, we have that  $\|\trace(A_{\varphi^i})\| \geq r$. 
\end{observation}

\subsection{Enfeoffing vertex subgraphs}

Our goal in this section is to prove the following proposition.  

\begin{prop}\label{prop3} In the notation above, suppose that $\sigma = I_n$. Then for any group-like vertex $v$ of $\CS \varphi$, there  exists a finite cover $\pi_0: \Gamma_0 \to \Gamma$ such that the lift of $\CT_v$ in $\Gamma_0$ in enfeoffed, or a finite cover to which $\varphi$ lifts where its homological action has infinite order. 
\end{prop}

\nid We begin by giving some definitions, and recalling standard facts about nilpotent groups. 
\subsubsection{Assigning subgroups to subgraphs}  Let  $\gamma = \eta_1 \ldots \eta_k$ be a cycle in $\CT$ originating at the vertex $\eta$. The cycle $\gamma$ corresponds to a path $e_1 \ldots e_s \nu^{\pm 1}$ in $\varphi^k(\eta)$. Since $\gamma$ is a cycle, the last edge of this path is $\nu^{\pm 1}$. Thus, $e_1 \ldots e_s \in \pi_1(\Gamma, \beta)$.  Define the \emph{group element} of $\gamma$ to be $\mathfrak{G}(\gamma) = e_1 \ldots e_s \in \pi_1(\Gamma_0, \beta_0)$ if the last edge is $\eta$ and $e_1 \ldots e_s \eta^{-1}$ otherwise . Note that the image of $\mathfrak{G}(\gamma)$ in $H_1(\Gamma)$ is $\mathfrak{t}(\gamma)$.

Now let $\CC$ be a collection of cycles in $\CT$. Define  $$\mathfrak{G}(\CC) = \langle \FG(\gamma) | \gamma \in \CC \rangle \leq \pi_1(\Gamma, \beta)$$

For every  $\nu \in V(\CT)$, and any $g \in \FG(\CC)$, define $\CC(\nu, g, +)$ to be the set of all $\gamma \in \CC$ originating at $\nu$ such that $\mathfrak{s}(\gamma) = 1$, $\FG(\gamma) = g$.   Define $\CC(\nu, g, -)$ similarly. We say that $\FG(\CC)$ is \emph{non degenerate} if there exist  $\nu, g$ such that $\# \CC(\nu, g, +) \neq \# \CC(\nu, g, -)$. What this means is that there are some cycles which correspond to sub words that are not cancelled in $F_n$.  Note that if $v$ is a group-like vertex of $\CS \varphi$, and $\CC_v$ is the collection of simple cycles in $\CT_v$ then $\FG(\CC_v)$ is non-degenerate.

\subsubsection{Nilpotent quotients of free groups} Let $F = F_n$. The \emph{lower central series of F} is the the sequence of  subgroups given by the recursive definition $F^{(0)} = F$, $F^{(i+1)} = [F, F^{(i)}]$. We denote by  $N_i = F/F^{(i)}$. Let $L_i = F^{(i-1)}/ F^{(i)}$. We require the following standard facts: 

\begin{enumerate} 
\item The groups $F^{(i)}$ are characteristic subgroups of $F$. 
\item The groups $N_i$ are nilpotent. 
\item The groups $L_i$ are finitely generated torsion free abelian groups, and furthermore $L_i = Z(N_i)$. If $S$ is a generating set, then the set of elements of the form $[a_1, \ldots, a_i]$ (where $a_1, \ldots, a_i \in S$, and $[a_1, \ldots, a_i] = [a_1, [a_2, [a_3, \ldots]]]$) generate $L_{i+1}$. 
\item Let $f \in \tup{Aut}(F)$ be an automorphism that acts trivially on $L_1$. Then $f$ acts trivially on $L_i$ for every $i$. 
\end{enumerate}

\subsubsection{Proof of Proposition \ref{prop3}} We are now ready to proceed with the proof. 
\begin{proof}
Let $v$ be a group-like vertex of $\CS \varphi$, and let $\CT_v$ be its vertex subgraph. Note that if $v = 0$, then $v$ is already enfeoffed (this follows directly from the fact that $\sigma = I_n$. Suppose that $v \neq 0$. 

Fix $i \geq 0$. Let $\pi_i: F_n \to N_i$ be the quotient map onto the $i^{th}$ nilpotent quotient of $F_n$. For any $k \geq 0$, let $\mathfrak{t}_k[\CT_v, N_i] = \sum \mathfrak{s}(\gamma)\pi_i \circ \FG(\gamma) \in \BZ[N_i]$, where the sum is taken over all cycles of length $k$ in $\CT_v$. We say that $\CT_v$ is \emph {$i^{th}$ level nilpotent enfeoffed} if $\mathfrak{t}_k[\CT_v, N_i] \neq 0$ for infinitely many $k$. 

\begin{lemma}\label{nilenf} There exists an $i$ such that $\CT_v$ is $i^{th}$ level nilpotent enfeoffed.   
\end{lemma}

\begin{proof}
Since $\CT_v$ is a vertex graph, for any two cycles $\gamma_1, \gamma_2$ of length $k$ in $\CT_v$ we have that $\pi_1 \circ \FG(\gamma_1) = \pi_1 \circ \FG(\gamma_2)$. Pick $x \in F_n$  whose image in $H_1(F_n)$ is $v$. For any $j$, denote  $h_{2,j} = \pi_2 (x^{-1} f^j(x)) \in  L_2$. For any  cycle $\gamma$ in $\CT_v$ of length $k$, we can write $$\pi_2 \circ \mathfrak{g} (\gamma) = \mathfrak{t}^{(2)}(\gamma)\prod_{j=1}^k \pi_2(x) h_{2,k-j}$$

\nid for some $\mathfrak{t}^{(2)}(\gamma) \in L_2$. By the definition of the action of $\varphi$, the fact that $L_2$ is central in $N_2$, and the fact that $f$ acts trivially on $L_2$, we have that  for any cycle $\gamma$ in $\CT_v$ of length $k$ that is a concatenation of simple cycles $\gamma_1, \ldots, \gamma_r$ based at that same point as $\gamma$ we have that: $\pi_2 \circ \mathfrak{g} (\gamma) = \mathfrak{t}^{(2)}(\gamma)\prod_{j=1}^k \pi_2(x) h_{2,k-j}$, where  $ \mathfrak{t}^{(2)}(\gamma) = \prod_{i=1}^r  \mathfrak{t}^{(2)}(\gamma_i)$. Thus, if we set $$\mathfrak{t}^{(2)}_k[\CT_v] = \sum \mathfrak{s}(\gamma) \mathfrak{t}^{(2)}(\gamma) \in \BZ[L_2]$$

\nid where the sum is taken over all based cycles of length $k$ in $\CT_v$, we have that:
$$\mathfrak{t}^{(2)}_k[\CT_v, N_i] =  \mathfrak{t}_k[\CT_v, N_i] \prod_{j=1}^k \pi_2(x) h_{2,k-j}$$

In contrast to the expressions of the form $\mathfrak{t}(\gamma)$ that we used in previous sections, the function $\mathfrak{t}^{(2)}(\gamma)$ is not necessarily invariant under a cyclic reordering of $\gamma$. However, by the additivity property discussed above, it is invariant under cyclic reorderings that preserve the basepoint. 

We define a new function $\mathfrak{bft}^{(2)}$ (or basepoint free translation) that assigns to each cycle $\gamma$ an element of $L_2^{V(\CT_v)}$ in the following way: for a vertex $\eta$ of $\CT_v$ and a cycle $\gamma$, define $(\mathfrak{bft}^{(2)})_\eta$ to be $0$ if $\gamma$ does not pass through $\eta$ and $\mathfrak{t}^{(2)}(\gamma_\eta)$ otherwise, where $\gamma_\eta$ is a cyclic reordering of $\gamma$ so that it is based at $\eta$. 

The function $\mathfrak{bft}^{(2)}$ is an additive function on cycles  to a torsion free, finitely generated abelian group whose value does not depend on cyclic reordering.   By the same reasoning as in section \ref{subord}, we have a (not necessarily proper) extremal subgraph $\CT_v^{(2)} \subset \CT_v$ such that for any cycle $\gamma$ in $\CT_v^{(2)}$ of length $k$, $\mathfrak{bft}^{(2)}(\gamma) = kv^{(2)}$ for some $v^{(2)} \in L_2^{V(\CT_v)} \otimes \BQ$,  and such that if $\CT_v^{(2)}$ is $i^{th}$ level nilpotent enfeoffed then so is $\CT_v$. By ignoring non-group like vertices, we can assume this graph is group-like. 

Since $kv \in  L_2^{V(\CT_v)}$ whenever $k$ is the length of a cycle in $\CT_v^{(2)}$, there is an integer $l$ such that $lv \in \CT_v^{(2)}$, and all cycles in $\CT_v^{(2)}$ have length divisible by $l$. Now denote  $h_{3,j} = \pi_3 (x^{-l} f^{jl}(x^l))$. Using the same reasoning as in the previous paragraph, we can write for any simple loop $\gamma$ in $\CT_v^{(2)}$, $$\pi_3 \circ \mathfrak{g} (\gamma) = \mathfrak{t}^{(3)}(\gamma)\prod_{j=1}^{k/l} \pi_3(x^{jl}) h_{3,k-j}$$
for some $t^{(3)}(\gamma) \in L_3$. Proceeding exactly as before, we find a (not necessarily proper) extremal subgraph $\CT^{(3)}_v \subset \CT_v^{(2)}$ where for any loop $\gamma$, $\mathfrak{t}^{(3)}(\gamma) = k v^{(3)}$ where $k$ is the length of $\gamma$ and $v^{(3)} \in L_3^{V(\CT_v^{(2)})}\otimes \BQ$. By construction, if $\CT^{(3)}_v$ is $i^{th}$ level nilpotent enfeoffed then so is $\CT_v^{(2)}$.

Proceeding in this manner, we get a sequence of graphs $\ldots \subset \CT^{(4)}_v \subset \CT^{(3)}_v \subset \CT^{(2)}_v \subset \CT_v$. Since $F_n$ is residually nilpotent, there exists an $i$ such that the subgroup of $F_n$ generated by the group elements assigned to all the cycles in $\CT_v^{(i)}$ based at any given vertex of this graph is cyclic. Indeed, otherwise if there were always two such cycles whose group elements did not form a cyclic group then by residual nilpotence one could pass to a further proper subgraph. Since $\CT_v$ is a finite graph, this process has to terminate. Since all the group elements of the loops in this group must have the same sign (this follows from the fact that $v \neq 0$, so they must all be positive powers of some fixed element), we have that $\CT_v^{(i)}$ is $i^{th}$ level nilpotent enfeoffed. 

\end{proof}

\begin{lemma} \label{nilpup}
Suppose that $\CT_v$ is the extremal graph of the group like vertex $v$, and that $\CT_v$ is $i^{th}$ level nilpotent enfeoffed. Then there is a characteristic cover $\pi_0: \Gamma_0 \to \Gamma$ where the lift of $\CT_v$ is $(i-1)^{th}$ level nilpotent enfeoffed. 
\end{lemma}

\begin{proof} Let $a_1, \ldots, a_i \in F_n$, and let $p \in \BN$. Then a calculation shows that $$[a_1^p, \ldots, a_i^p] \equiv p^i [a_1, \ldots, a_i] \tup{ } (\tup{mod } F_n^{(i+1)})$$
(see for example \cite{Had2}). Thus, if we set $K_p < F_n$ to be the kernel of the map $F_n \to H_1(F_n, \BZ/p\BZ)$, then the image of $K_p^{(i-1)}$ in $L_{i-1}$ is $p^{i-1}L_{i-1}$. Note that this cover is characteristic. 

\nid Using the notation of Lemma \ref{nilenf}, let $$\mathfrak{bft}^{(i)}_k[\CT_v^{(i)}] = \sum \mathfrak{s}(\gamma) \mathfrak{bft}^{(i)}(\gamma)  $$ 
where the sum is taken over all loops of length $k$ in $\CT_v^{(i)}$. Let $\mathfrak{bft}^{(i)}_k[\CT_v^{(i)},p]$ be the sum we get by taking only the summand of $\mathfrak{bft}^{(i)}_k[\CT_v^{(i)}]$ that lie in $p^{i-1}L^{V(\CT_v^{(i)})}_{i-1}$.  If we show that there exists a $p$ such that for infinitely many values of $k$, $\mathfrak{bft}^{(i)}_k[\CT_v^{(i)},p] \neq 0$, then we will have that $\CT_v$ is $(i-1)^{th}$ level enfeoffed in $K_p$. 

Denote $M =  L_{i-1}^{V(\CT_v^{(i)})}$. By construction, the map $\mathfrak{bft}^{(i)}$ gives a homomorphism from $H_1(\CT_v^{(i)})$ to $ M$. It can thus be extended to a homomorphism from $C_1(\CT_v^{(i)}) \to M$. Thus, if $\CT_v^{(i)}$ has $m$ vertices, we can construct a matrix $A \in GL_m(\BZ[M])$ such that $\mathfrak{t}^{(i)}_k[\CT_v^{(i)}] = \trace A^k$, for any $k$. 

Let $Q$ be the fraction ring of $Z = \BZ[M]$. Fix some algebraic closure of $Q$. Let $\rho \in \BZ[X]$ be the characteristic polynomial of $A$,  let $K$ be its splitting field, and let $\alpha_1, \ldots, \alpha_m$ be its roots.  Let $r = \rank M$, and let $\BT = (S^1)^r$ be its dual group. 

We have a map $\mathfrak{e}: \BT \to \BC_{\leq m}[X]$ given by sending $\xi \in \BT$ to the polynomial obtained by applying $\xi$ to the coefficients of $\rho$. The map  $\xi$ is coordinate-wise real analytic.  Let $$\CX  = \{(\xi, z) \in \BT \times \BC : \mathfrak{e}(\xi)[z] = 0 \}$$ 

Let $\pi_\BT: \CX \to \BT$ be projection onto the first coordinate. For any $\xi$, counting with multiplicity, $\#\pi_\BT^{-1}(\xi) = m$. We say that a function $\psi: \BT \to \BC$ is a \emph{root} if for every $\xi \in \BT$, $(\xi, \psi(\xi)) \in \CX$.\label{roots}Since $\CX$ is defined by real analytic equations, we can choose $m$ continuous roots, $\psi, \ldots, \psi_m$ such that for each $\xi$, $\psi_1(\xi), \ldots, \psi_m(\xi)$ is the set of all solutions to $\mathfrak{e}(\xi)$, counted with multiplicity. Let $\wh{\psi_1}, \ldots, \wh{\psi_m} \in L^1[M]$ be the inverse Fourier transforms of these roots. This gives an embedding of the ring $R = Z[\alpha_1, \ldots, \alpha_m]$ into $L^1[H]$. Note that $K$ is isomorphic to the quotient ring of $R$. By choosing such an embedding, we will identify $\alpha_1, \ldots, \alpha_m$ with specific elements of $L^1[M]$. 

For any $p, k$, we have that $\mathfrak{t}^{(i)}_k[\CT_v^{(i)},p] = \sum_{j=1}^m [\alpha_j^k]_{p^{i-1}}$, where $[\cdot]_{p^{i-1}}$ denotes taking only the summands in $p^{i-1}M$.  If we set $I_p$ to be the indicator function of $p^{i-1}M$, we can write $\mathfrak{t}^{(i)}_k[\CT_v^{(i)},p] = \sum_{j=1}^m \alpha_j^k \cdot I_p$. The inverse Fourier transform of $I_p$ is the distribution $\sum_{x \in R_p} \delta_x$, where $\delta_x$ is a Dirac distribution supported at $x$, and $R_p$ is the set of all points in $\BT$ whose coordinates are all $p^{i-1}$ roots of unity.  Using the fact that the Fourier transform is an isometry, we get that:

$$\mathfrak{t}^{(i)}_k[\CT_v^{(i)},p] = \sum_{j=1}^m \sum_{x \in R_p} \psi_j^k(x)$$

By the Vandermonde theorem, the expression on the right hand is nonzero for infinitely many values of $k$ if and only if $\psi_j|_{R_p} \neq 0$ for every $j$. Since $A$ is not nilpotent (this follows from the fact that $\CT_v^{(i)}$ is $i^{th}$ level nilpotent enfeoffed), we have that one of the $\psi_j$'s is a nonzero function. Since all these functions are continuous, and union of all $R_p$'s is dense in $\BT$, we can find a $p$ and a $j$ such that $\psi_j|_{R_p} \neq 0$. We  deduce that $\CT_v$ is $(i-1)^{th}$ level nilpotent enfeoffed in $K_p$. 

\end{proof}

By Lemma \ref{nilenf}, the graph $\CT_v$ is $i^{th}$ level nilpotent enfeoffed for some $i$. By repeated application of Lemma \ref{nilpup}, we can either find a cover where the homological action has infinite order, or one where the lift of $\CT_v$ is $1^{st}$ level nilpotent enfeoffed. Since this is the same as our original definition of being enfeoffed, we are done. 

\end{proof}

\subsection{Proof of Theorem \ref{theorem2}}
If $\sigma$ has infinite order, we are done. Otherwise, by replacing $f$ with a power of itself we can assume that $\sigma$ is trivial and every vertex of $\CS f$ is a group-like vertex of $\CS \varphi$.  By Proposition \ref{prop3}, for each such vertex we can find a cover where it is enfeoffed. Byperforming this process for every group-like vertex, we find a cover where every group-like vertex of $\CS \varphi$ is enfeoffed. Since we have only taken abelian covers at each step, this cover is solvable. The result now follows from Corollary \ref{cor1}.

\end{document}